\documentclass{amsart}

\usepackage{amsfonts,amssymb,amsmath,enumerate,verbatim,color}

\usepackage[colorlinks,pagebackref]{hyperref}

\usepackage[all]{xy}
\xyoption{curve}
\SelectTips{eu}{}

\newcommand\bschi{{\boldsymbol \chi}}
\newcommand\bsG{{\boldsymbol G}}
\newcommand\bsf{{\boldsymbol f}}
\newcommand\bsg{{\boldsymbol g}}

\newcommand\BN{{\mathbb N}}
\newcommand\BQ{{\mathbb Q}}
\newcommand\BR{{\mathbb R}}
\newcommand\BZ{{\mathbb Z}}

\newcommand\Coker{\operatorname{Coker}}
\newcommand{\cx}{\operatorname{cx}}

\newcommand\depth{\operatorname{depth}}

\newcommand\grade{\operatorname{grade}}
\newcommand\edim{\operatorname{edim}}
\newcommand\mult{\operatorname{mult}}
\newcommand\ext{\operatorname{Ext}}
\newcommand\tor{\operatorname{Tor}}
\newcommand{\Tor}[4]{\operatorname{Tor}_{#1}^{#2}(#3,#4)}
\newcommand{\Ext}[4]{\operatorname{Ext}^{#1}_{#2}(#3,#4)}

\newcommand{\CE}[3]{\operatorname{E}_{#1}^{#2,#3}}
\newcommand{\HE}[3]{\operatorname{E}^{#1}_{#2,#3}}

\newcommand\fm{{\mathfrak m}}
\newcommand\fn{{\mathfrak n}}
\newcommand\fp{{\mathfrak p}}
\newcommand\fq{{\mathfrak q}}

\newcommand\Hom{\operatorname{Hom}}
\newcommand\colim{\operatorname{colim}\,}

\newcommand{\hh}[1]{\operatorname{H}(#1)}
\newcommand{\HH}[2]{\operatorname{H}_{#1}(#2)}

\newcommand\Image{\operatorname{Im}}
\newcommand\Ker{\operatorname{Ker}}

\newcommand\length{\operatorname{length}}
\newcommand{\les}{\leqslant}
\newcommand{\ges}{\geqslant}
\newcommand{\lra}{\longrightarrow}
\newcommand{\ov}{\overline}

\newcommand{\wh}{\widehat}
\newcommand{\wt}{\widetilde}

\newcommand{\pdim}{\operatorname{proj\,dim}}

\newcommand{\idim}{\operatorname{inj\,dim}}

\newcommand{\rank}{\operatorname{rank}}

\newcommand{\shift}{{\sf\Sigma}}
\newcommand\Spec{\operatorname{Spec}}
\newcommand\mspec{\operatorname{Max}}

\newcommand{\syz}[3][Q]{{\Omega^{#1}_{#2}(#3)}}

\newcommand{\xla}{\xleftarrow}
\newcommand{\xra}{\xrightarrow}

\newtheorem{theorem}{Theorem}[section]

\newtheorem{proposition}[theorem]{Proposition}
\newtheorem{lemma}[theorem]{Lemma}
\newtheorem{corollary}[theorem]{Corollary}

\theoremstyle{definition}
\newtheorem{example}[theorem]{Example}
\newtheorem{chunk}[theorem]{}

\theoremstyle{remark}
\newtheorem{remark}[theorem]{Remark}
{}
{}
{}
{}
{}

\newtheorem*{Question}{Question}

\numberwithin{equation}{theorem}

\begin{document}

\title[Persistence of homology]{Persistence of homology over\\ commutative noetherian rings}

\author[Avramov]{Luchezar L.~Avramov}
\address{Department of Mathematics,
University of Nebraska, Lincoln, NE 68588, U.S.A.}
\email{avramov@math.unl.edu}

\author[Iyengar]{Srikanth B.~Iyengar}
\address{Department of Mathematics,
University of Utah, Salt Lake City, UT 84112, U.S.A.}
\email{iyengar@math.utah.edu}

\author[Nasseh]{Saeed Nasseh}
\address{Department of Mathematical Sciences,
Georgia Southern University, Statesboro, GA 30460, U.S.A.}
\email{snasseh@georgiasouthern.edu}

\author[Sather-Wagstaff]{Sean K.~Sather-Wagstaff}
\address{Department of Mathematical Sciences,
Clemson University,
O-110 Martin Hall, Box 340975,
Clemson, SC 29634, U.S.A.}
\email{ssather@clemson.edu}

\thanks{Research partly supported by NSF grants DMS-1103176 (LLA) and DMS-1700985 (SBI)}

\date{\today}


\begin{abstract}
We describe new classes of noetherian local rings $R$ whose finitely generated  modules $M$ have the property that $\Tor iRMM=0$ for $i\gg0$ implies that $M$  has finite projective dimension, or $\Ext iRMM=0$ for $i\gg0$ implies that  $M$ has finite projective dimension or finite injective dimension.
  \end{abstract}

\maketitle

\section*{Introduction}
This work concerns homological  dimensions of modules over commutative noetherian rings. Over such a ring $R$, a finite (meaning, finitely generated) $R$-module $M$ has finite projective dimension if $\Tor iRMN=0$ for $i\gg 0$  for each finite $R$-module $N$.

We say that $R$ is \emph{homologically persistent}, or \emph{$\tor$-persistent}, if every finite  $R$-module $M$ for which $\Tor iRMM=0$ for $i\gg 0$ satisfies $\pdim_RM<\infty$.

Every regular ring is $\tor$-persistent for then $\pdim_RM$ is finite by fiat. Results of Avramov and Buchweitz~\cite{AB} imply  that the same holds for locally complete intersection rings.  The work reported in our paper was sparked by that of  \c{S}ega~\cite[pp.~1266]{Se2} who verifies it for  certain classes of local rings, and says, in paraphrase:  ``The author does not know any examples  of rings that are \emph{not} $\tor$-persistent.'' Neither do we, so we ask:

\begin{Question}
\label{qu:tp}
Is every ring $\tor$-persistent?
\end{Question}

The purpose of this paper is to present some new non-trivial examples of Tor-persistent rings. It is not hard to prove that this property can be detected locally---see Proposition~\ref{pr:bass-murthy}---so for the rest of the introduction $R$ will be a local ring.

A central result of our work is that $R$ is Tor-persistent if its completion has a \emph{deformation} $Q$ (that is to say, it is of the form $Q/(\bsf)$ where $\bsf=f_{1},\dots,f_{c}$ is a $Q$-regular sequence) that satisfies any one of the following conditions:
\begin{itemize}
\item
$\edim Q-\depth Q\le3$.
 \item
$\edim Q-\depth Q=4$ and $Q$ is Gorenstein.
 \item
$\edim Q-\depth Q=4$ and $Q$ is Cohen-Macaulay, almost complete intersection,
with $\frac12\in Q$.
 \item
$Q$ is one link from a complete intersection.
 \item
$Q$ is two links from a complete intersection and is Gorenstein.
 \item
$Q$ is Golod.
  \item
$\mult Q\le7$ and $Q$ is Cohen-Macaulay.
  \end{itemize}

This result is contained in Theorem~\ref{thm:mainTor}. In many of the cases considered above, we deduce the Tor-persistence of $R$ by verifying that $Q$ has the stronger property that if $\Tor iQMN=0$ for $i\gg 0$, then $M$ or $N$ has finite projective dimension over $Q$. This property, which we call \emph{Tor-friendliness}, has been studied by Jorgensen~\cite{Jo}, Huneke, \c{S}ega, and Vraicu~\cite{HSV}, \c{S}ega~\cite{Se1}, Nasseh and Yoshino~\cite{NY}, Gosh and Puthenpurakal~\cite{GP}, Lyle and Monta\~no~\cite{LM}, among others, and our work extends some of their results.  Tor-friendliness of $R$ can also be recast as the statement that each finite $R$-module $M$ of infinite projective dimension is a test module for finite projective dimension, in the sense of Celikbas, Dao, and Takahashi~\cite{CDT}.

To prove Theorem~\ref{thm:mainTor}, we draw on a panoply of results concerning the homological properties of the classes of rings appearing in the list above, including structure theorems for their Koszul homology algebra proved by Kustin~\cite{Ku1} in collaboration with Miller~\cite{KM1,KM2,KM3},  Jacobsson and Miller~\cite{JKM}, Avramov and Miller~\cite{AKM}, and Palmer Slattery \cite{KP}. To the same end, we also establish a number of general statements for recognizing Tor-persistence, and for tracking that property along homomorphisms of rings. Notable among these is the result below which covers the case of complete intersection rings.

\smallskip

\emph{Assume the residue field of $R$ is algebraically closed and that $R$ has a deformation $Q$ such that $0$ is the only finite $Q$-module with constant Betti numbers. If $Q$ is $\tor$-persistent, then so is $R$.}

\smallskip

This is the content of Theorem~\ref{thm:rigidT}. To our vexation we have been unable to eliminate the hypothesis on $Q$-modules with constant Betti numbers, except in the  case when the generators of the ideal $\Ker(Q\to R)$ can be chosen in $\fn\setminus \fn^{2}$, where $\fn$ is the maximal ideal of $Q$; see Proposition~\ref{prop:SegaFriendly}(1). This latter result, and some other intermediate ones, have also been proved recently by Celikbas and Holm~\cite{CH}.

In Sections \ref{sec:Vanishing of Ext} and \ref{sec:ARP} we turn to cohomological analogs of Tor-persistence. This comes in two flavors: We say $R$ is Ext-\emph{persistent} if for a finite $R$-module $M$ satisfies $\Ext iRMM=0$ for $i\gg 0$ only when the projective dimension or the injective dimension of $M$ is finite. While not every ring is Ext-persistent (consider rings with nontrivial semidualizing modules), we prove that the class of rings $R$ in Theorem~\ref{thm:mainTor} are, which implies in particular that there are no nontrivial semidualizing modules, or even complexes, over such rings; see  Corollary~\ref{cor:mainExt}.

The ring $R$ has the \emph{Auslander-Reiten property} if a finite $R$-module $M$ satisfies $\Ext iRM{M\oplus R}=0$ for $i\gg 0$ only when its projective dimension is finite. The nomenclature is motivated by a conjecture of Auslander and Reiten that asserts that every $R$ has this property. Once again, the class of rings $R$ in Theorem~\ref{thm:mainTor} have the Auslander-Reiten property; see Corollary~\ref{cor:AR}.

  \section{Persistence of homology}
    \label{S:Strong-u-rings}
In this section we assemble sufficient conditions that ensure that a ring is $\tor$-persistent.
We first recast the definitions in terms of complexes. Throughout, $R$ denotes a commutative noetherian ring.

  \begin{chunk}
    \label{ch:homdim}
An \emph{$R$-complex} is a family of $R$-linear maps $\partial^U_i\colon U_i\to U_{i-1}$ with $\partial^U_{i-1}\partial^U_i=0$.
We set
  \[
\inf U=\inf\{n\in\BZ\mid U_n\ne0\}
  \quad\text{and}\quad
\sup U=\sup\{n\in\BZ\mid U_n\ne0\}
  \]
and say that $U$ is \emph{bounded below} (respectively,  \emph{above}) if $U=0$ or $\inf U$ (respectively, $\sup U$) is finite.
When both conditions hold $U$ is said to be \emph{bounded}. It is \emph{finite} if the $R$-module $\bigoplus_{i\in\BZ}U_i$ is finite;
in particular, then $U$ is bounded.

The homology of $U$ is the complex $\hh U$ with $\HH iU$ in degree $i$ and $\partial^{\hh U}=0$.
\end{chunk}

We write $\Spec R$ (respectively, $\mspec R$) for the set of prime ideals (respectively, maximal ideals) of $R$, with the Zariski topology.

\begin{chunk}
\label{ch:perfect}
An $R$-complex is said to be \emph{perfect} if it is quasi-isomorphic to a bounded complex of finite
projective $R$-modules.  This property can be verified homologically: An $R$-complex $U$ is perfect if and only if $\hh U$ is degreewise finite and bounded below and  $\Tor{}R{R/\fm}U$ is bounded for each $\fm$ in $\mspec R$.

Indeed, by \cite[5.5.F]{AF} the boundedness implies that the $R_\fm$-complex $U_\fm$ is perfect, and then \cite[4.1]{AIL} yields that $U$ is perfect. We refer to \cite{AF} for the construction of Ext and Tor for $R$-complexes and related notions.
  \end{chunk}

  \begin{proposition}
  \label{pr:persistent}
The following conditions on a ring $R$ are equivalent.
  \begin{enumerate}[\quad\rm(i)]
 \item
The ring $R$ is $\tor$-persistent.
 \item
Each $R$-complex $U$ with $\hh U$ finite and $\Tor{}RUU$ bounded is perfect.
\item
Each $R$-complex $U$ with $\hh U$ of finite length and $\Tor{}RUU$ bounded is perfect.
  \end{enumerate}
  \end{proposition}

  \begin{proof}
The implications (ii)$\implies$(i) and (ii)$\implies$(iii) are tautologies. In the rest of the proof we repeatedly use \ref{ch:perfect}, and without comment.

(i)$\implies$(ii).
Replacing $U$ by an appropriate resolution, we may assume that $U$  is semiprojective with each $U_{i}$ a finite free $R$-module. Set $s=\sup\hh U$ and consider the canonical exact sequence of $R$-complexes
\begin{equation}
\label{eq:syzygy}
0\lra U_{<s}\lra U \lra U_{\ges s}\lra 0\,.
\end{equation}
Observe that $\shift^{-s}U_{\ges s}$ is a free resolution of the finite $R$-module  $M=\HH s{U_{\ges s}}$. Thus the exact sequence above implies isomorphisms
\[
\Tor iRMM \cong \Tor {i+2s}RUU \quad\text{for $i\ge 2s+1$}.
\]
Since $\Tor{}RUU$ is bounded it follows that so is $\Tor{}RMM$, and hence, by hypothesis, $M$ is perfect. Since the complex $U_{<s}$ is evidently perfect, it then follows from \eqref{eq:syzygy} that $U$ is perfect as well, as desired.

(iii)$\implies$(ii).
Let $U$ be a complex as in (ii); we may assume it is semiprojective. It suffices to verify that $\Tor{}R{R/\fm}U$, that is to say, $\hh{R/\fm\otimes_{R}U}$, is bounded for each maximal ideal $\fm$ of $R$.  Let $K$ be the Koszul complex on a finite generating set for $\fm$. A standard computation shows that $\hh{K\otimes_{R}U}$ has finite length. Moreover, $K\otimes_{R}U$ is semiprojective, so one gets the first isomorphism below
\begin{align*}
\Tor{}R{K\otimes_{R}U}{K\otimes_{R}U}
	&\cong \hh{(K\otimes_{R}U)\otimes_{R}(K\otimes_{R}U)} \\
	& \cong \hh{(K\otimes_{R}K)\otimes_{R}(U\otimes_{R}U)}
\end{align*}
The second one is due to the associativity of tensor products.  Note that $K\otimes_{R}K$ is a bounded complex of free $R$-modules.
It follows that $\hh{(K\otimes_{R}K)\otimes_{R}(U\otimes_{R}U)}$ is bounded, along with $\hh{U\otimes_{R}U}$. In view of the
preceding isomorphisms and our present hypothesis, $K\otimes_{R}U$ is perfect. Therefore, $\hh{k\otimes_{R}(K\otimes_{R}U)}$
is bounded when $k=R/\fm$. From associativity and the K\"unneth isomorphism we get
\begin{align*}
\hh{k\otimes_{R}(K\otimes_{R}U)}
	&\cong \hh{(k\otimes_{R}K)\otimes_{k}(k\otimes_{R}U)} \\
	&\cong \hh{k\otimes_{R}K} \otimes_{k}\hh{k\otimes_{R}U}
\end{align*}
Since $\hh{k\otimes_{R}K}\ne 0$, it follows $\hh{k\otimes_{R}U}$ is bounded, so $U$ is perfect.
\end{proof}

As a first application, we obtain a change of rings property for homological persistence. A homomorphism of rings $R\to S$ is of \emph{finite flat dimension} if $S$ has finite flat dimension when viewed as an $R$-module by restriction of scalars. In the next result the hypothesis on the maximal ideals holds when $R$ and $S$ are local rings and the homomorphism is local.

\begin{proposition}
\label{prop:descent}
Let $R\to S$ be a homomorphism of finite flat dimension such that the induced map $\mspec S\to \mspec R$ is surjective.

If $S$ is $\tor$-persistent, then so is $R$.
\end{proposition}

\begin{proof}
Let $M$ be a finite $R$-module with $\Tor {}RMM$ bounded.  Let $P$ be a semiprojective resolution of $M$  and set $U=S\otimes_{R}P$. Evidently, the $S$-complex $U$ is semiprojective, which justifies the first isomorphism below.
\begin{align*}
\Tor{}SUU
	&\cong \hh{(S\otimes_{R}P)\otimes_{S}(S\otimes_{R}P))} \\
	&\cong \hh{(S\otimes_{R} P) \otimes_{R}P)} \\
	& \cong \hh{S\otimes_{R} (P\otimes_{R}P)}
\end{align*}
The other isomorphisms are standard.  As $\hh{P\otimes_{R}P}$ is isomorphic to $\Tor{}RMM$, it is bounded.  Given this and the hypothesis that the flat dimension of $S$ over $R$ is finite, it follows that the homology of $S\otimes_{R} (P\otimes_{R}P)$, that is to say,  $\Tor{}SUU$, is bounded. By the same token, the $S$-module $\hh U$ is finite. Since $S$ is $\tor$-persistent, it follows that the $S$-complex $S\otimes_{R}P$ is perfect; see Proposition~\ref{pr:persistent}.

Pick a maximal ideal $\fm$ of $R$.  By hypothesis, there is a maximal $\fn$ of $S$, so that $\fn\cap R=\fm$.  Set $k=R/\fm$ and $l=S/\fn$.
We then have isomorphisms
\[
l\otimes_{k}\Tor{}RkM\cong\Tor{}RlM \cong \hh{l\otimes_{R}P} \cong \hh{l\otimes_{S}(S\otimes_{R}P)}
\]
where $R$ acts on $l$ through the composed ring homomorphism $R\to k\to l$. Since $S\otimes_{R}P$ is perfect, the graded module on the right, and hence $\Tor{}RkM$, is bounded.  As $\fm$ was arbitrary, \ref{ch:perfect} yields that $M$ is perfect as an $R$-complex.
\end{proof}

\begin{proposition}
\label{ch:persistent-complete}
Let $I$ be an ideal in the Jacobson radical of $R$, and let $\wh R$ denote the $I$-adic completion of $R$.

The ring $R$ is $\tor$-persistent if and only if so is $\wh R$.
\end{proposition}

\begin{proof}
The completion map $R\to \wh R$ is flat and the assignments $\fn\mapsto \fn\cap R$ and $\fm\mapsto\fm\wh R$
yield inverse bijections $\mspec \wh R\leftrightarrow \mspec R$.  Now Proposition~\ref{prop:descent}
shows that if $\wh R$ is $\tor$-persistent, then so is $R$.

Conversely, let $V$ be an $\wh R$-complex with homology of finite length and $\Tor{}{\wh R}VV$ bounded. When an $\wh R$-module
$L$ has finite length, then it has the same length over $R$ and the canonical $\wh R$-linear map $\wh R\otimes_RL\to L$ is bijective.
Thus, the map $\wh R\otimes_RV\to V$ is a quasi-isomorphism,  so we have isomorphisms
\[
\wh R\otimes_R\Tor{}RVV\xra{\cong} \Tor{}{\wh R}{\wh R\otimes_RV}{\wh R\otimes_RV}\xra{\cong} \Tor{}{\wh R}{V}{V}
\]
Since $R$ is $\tor$-persistent and $R\to\wh R$ is faithfully flat, it follows that $V$ is perfect as an $R$-complex. For every maximal ideal
$\fn$ of $\wh R$ and for $\fm=\fn\cap R$ the isomorphisms
\[
\Tor{}{\wh R}{\wh R/\fn}V =\Tor{}{\wh R}{\wh R/\fm \wh R}V \cong \Tor{}{\wh R}{(R/\fm)\otimes_{R}\wh R} V \cong \Tor{}R{R/\fm}V
\]
show that $\Tor{}{\wh R}{\wh R/\fn}V$ is bounded, so $V$ is perfect, by \ref{ch:perfect} .
\end{proof}

The import of the next result is that Tor-persistence is a local property.

\begin{proposition}
\label{pr:bass-murthy}
The ring $R$ is $\tor$-persistent if and only if $R_\fm$ is $\tor$-persistent for every $\fm\in\mspec R$.
\end{proposition}

\begin{proof}
Indeed, let $M$ be a finite $R$-module. If  $\pdim_{R_\fm}M_\fm$ is finite for every $\fm\in\mspec R$, then so is $\pdim_{R}M$; see Bass and Murthy \cite[4.5]{BM}, or \cite[5.1]{AIL}. It follows that when $R_{\fm}$ is $\tor$-persistent for each such $\fm$, then $R$ is $\tor$-persistent.

Assume $R$ is $\tor$-persistent and fix an $\fm\in \mspec R$. By Proposition~\ref{pr:persistent}, it suffices to prove that if $U$ is an $R_{\fm}$-complex with $\hh U$ of finite length and $\Tor{}{R_{\fm}}UU$ is bounded, then $U$ is perfect. Since the residue field at $\fm$ is finite as an $R$-module, $\hh U$ is finite also when viewed as an $R$-complex, via restriction of scalars along the localization $R\to R_{\fm}$. Moreover, the canonical map is an isomorphism:
\[
\Tor{}RUU \xra{\  \cong\ } \Tor{}{R_{\fm}}UU\,.
\]
Thus $\Tor{}RUU$ is bounded. Therefore $U$ is perfect as an $R$-complex and hence also as an $R_{\fm}$-complex, since localization induces a quasi-isomorphism $U\xra{\simeq} U_{\fm}$.
\end{proof}

\begin{chunk}
\label{ch:CST}
As usual, if $R$ is local then $\wh R$ denotes its completion in the adic topology of the maximal ideal.
Recall that the natural map $R\to\wh R$ is a faithfully flat ring homomorphism and that Cohen's Structure
Theorem yields an isomorphism $\wh R\cong P/I$ for some regular local ring $(P,\fp,k)$ and ideal $I$ contained
in $\fp^{2}$; any such isomorphism is called a \emph{minimal Cohen presentation} of $\wh R$.

We say that $R$ is \emph{locally complete intersection} if $\wh{R_{\fm}}$ has a minimal Cohen
presentation with ideal of relations generated by a regular sequence, for each $\fm\in\mspec R$.
   \end{chunk}

\begin{example}
\label{ex:lci-unbounded}
If $R$ is locally complete intersection, then $R$ is $\tor$-persistent.

Indeed, by Proposition \ref{pr:bass-murthy} we may assume $R$ is local.  Avramov and Buchweitz
\cite{AB} attached to every finite $R$-module $M$ a \emph{cohomological support variety}
$\operatorname{V}^*_R(M)$.  It is a closed subset of some projective space over $k$, and
\cite[6.1, 4.7, and 4.9]{AB} show that $\Tor{}RMM$ is bounded if and only if $\operatorname{V}^*_R(M)$
is empty.  By \cite[5.6(9) and 5.6(3)]{AB}, the latter is equivalent to $\Ext iRMk=0$ for $i\gg0$; that is, to $\pdim_RM<\infty$.
    \end{example}

\section{Deformations}
\label{sec:Deformations}

In this section $(R,\fm,k)$ denotes a local ring and $M$ and $N$ are finite $R$-modules.

\begin{chunk}
   \label{ch:def}
By a \emph{deformation} of $R$ to $Q$, we mean a surjective homomorphism $R\twoheadleftarrow Q$ from a local ring $Q$, with kernel generated by a $Q$-regular sequence; it is said to be \emph{embedded} if  $\edim Q=\edim R$.  If every embedded deformation of $R$ is bijective, then we say that $R$ has \emph{no embedded deformation}.
  \end{chunk}

By Proposition \ref{prop:descent}, if $R$ is $\tor$-persistent, then so is $Q$.  It has turned out to be surprisingly difficult to answer the question: Does the converse hold?  We have only been able to obtain positive  answers under additional assumptions. In the next result the hypothesis concerning the Betti numbers of $Q$  is particularly vexing.  Recall that the $n$th \emph{Betti number} of a finite module $L$ over a local ring  $(Q,\fn,k)$ can be defined by the equality $\beta^Q_n(L)=\rank_k\Tor nQLk$.

    \begin{theorem}
  \label{thm:rigidT}
Let $R\twoheadleftarrow Q$ be a deformation such that $k$ is algebraically closed and there is no nonzero
finite $Q$-module with constant Betti numbers.

The ring $R$ is $\tor$-persistent if and only if the ring $Q$ is.
    \end{theorem}

Another case, when the same conclusion holds, is given by Proposition \ref{prop:SegaFriendly}(1).

The proof of the theorem is presented in \ref{ch:PfrigidT}.  It utilizes a number of homological constructions,
which we proceed to review, starting with a spectral sequence.

  \begin{chunk}
    \label{ch:ssTor}
Let $R\twoheadleftarrow Q$ be a deformation and $\bsf=\{f_1,\dots,f_c\}$ a minimal generating
set of $I=\Ker(Q\twoheadrightarrow R)$.  By resolving $R$ over $Q$ by means of the Koszul
complex on $\bsf$ one gets isomorphisms $\Tor qQRN\cong N^{\binom cq}$ for every integer $q$, so
the standard change-of-rings spectral sequence with $\HE 2pq=\Tor pR{\Tor qQMR}N$ takes the form
  \begin{align}
    \label{eq:ssTor}
\HE 2pq=\Tor pRMN^{\binom cq}\implies\Tor{p+q}QMN
  \end{align}

It follows immediately that if $\Tor{}RMN$ is bounded, then so is $\Tor{}QMN$.
    \end{chunk}

Next we turn to the multiplicative structure of cohomology.  In this section it will be needed only for
$R$-modules, but for later use we describe its properties in the more general framework of $R$-complexes.

   \begin{chunk}
    \label{ch:AGP}
Let $U$ and $V$ be $R$-complexes.

Composition products $\circ$ turn $\Ext{}RUU$ and $\Ext{}RVV$ into $R$-algebras graded by cohomological degree, and endow $\Ext{}RUV$ with a structure of graded  bimodule on which $\Ext{}RUU$ acts on the right and $\Ext{}RVV$ on the left.

With $R\twoheadleftarrow Q$ and $\bsf$ as in \ref{ch:ssTor}, let $R[\bschi]$ be a polynomial ring in indeterminates
$\chi_1,\dots,\chi_c$ of cohomological degree $2$.  By \cite[2.7, p.\,700]{AS} there are homomorphisms
  \[
\Ext{}RVV\xla{\,\zeta_V\,} R[\bschi]\xra{\,\zeta_U\,}\Ext{}RUU
  \]
 of graded $R$-algebras, with images in the corresponding centers, satisfying
  \[
\zeta_V(\rho)\circ\xi=\xi\circ\zeta_U(\rho)
\quad\text{for all}\quad
\rho\in R[\bschi]
\quad\text{and}\quad
\xi\in\Ext{}RUV\,.
  \]
Thus, both maps give the same graded $R[\bschi]$-module structure on $\Ext{}RUV$.

The $R[\bschi]$-module $\Ext{}RUV$ is finitely generated if and only if $\Ext{}QUV$ is bounded;
see Avramov and Sun, \cite[5.1]{AS}.
    \end{chunk}

The last result is used to define cohomological support varieties; see Remark~\ref{ch:variety}. It is expedient to first recall  a measure of the asymptotic growth of minimal free resolutions, introduced in \cite{Av:vpd}; see also \cite[\S4.2]{Av:barca}.

\begin{chunk}
   \label{ch:les}
The \emph{complexity} of $M$ over $R$ is the number defined by the equality
  \[
\cx_RM=\inf\{d\in\BN\cup\{0\}\mid\beta^R_n(M)\le bn^{d-1} \text{ for some } b\in\BR \text{ and all }n\gg0\}.
  \]
In particular, $\cx_RM=0$ means that $M$ has finite projective dimension, while $\cx_RM=1$ means that
the sequence $(\beta^R_n(M))_{n\ges0}$ is bounded.
   \end{chunk}

 \begin{chunk}
    \label{ch:variety}
Let $R\twoheadleftarrow (Q,\fq,k)$ be a deformation and assume that $k$ is algebraically closed.

Set $I=\Ker(R\twoheadleftarrow Q)$, let $\bsf=\{f_1,\dots,f_c\}$ be a regular sequence generating $I$,
and write $\ov f$ for the image in $I/\fq I$ of $f\in I$.  Set $k[\bschi]=k\otimes_RR[\bschi]$, see
\ref{ch:AGP}, and identify $k[\bschi]$ and the ring of $k$-valued algebraic functions on $\ov I=I/\fq I$ by
mapping $1\otimes\chi_i$ to the $i$th coordinate function of the $k$-basis $\{\ov f_1,\dots,\ov f_c\}$ of
$\ov I$.  As in \cite[\S\,3]{Av:vpd}, let $V(Q,\bsf,M)$ denote the zero-set in $\ov I$ of the annihilator of
$\Ext{}RMk$ in $k[\bschi]$.

When $\pdim_QM$ is finite, \cite[3.12 and 3.11]{Av:vpd} yield the following equalities:
  \begin{align}
    \label{eq:variety1}
\cx_RM&=\dim V(Q,\bsf,M)
  \\
    \label{eq:variety2}
V(Q,\bsf,M)&=\{\ov f\mid f\in I\smallsetminus\fq I\text{ with }\pdim_{Q/(f)}M=\infty\}\cup\{0\}\,.
  \end{align}
Any $f$ in $I\smallsetminus\fq I$ can be extended to a minimal generating set for $I$. Hence such an
$f$ is not a zero divisor on $Q$, and the canonical map $Q/(f)\to R$ is a deformation.
    \end{chunk}

Contained in the preceding formulas is a criterion for finite projective dimension.

 \begin{proposition}
    \label{prop:criterion}
Let $R\twoheadleftarrow (Q,\fq,k)$ be a deformation such that $I=\Ker(Q\to R)$ is nonzero and $k$ is algebraically closed. When $M$ is a finitely generated $R$-module, $\pdim_{R}M$  is finite if and only if $\pdim_{Q/(f)}M$ is finite for each $f\in I\smallsetminus\fq I$.
    \end{proposition}

    \begin{proof}
The validity  of either hypothesis implies $\pdim_QM<\infty$:  This follows from $\pdim_{R}M<\infty$ because
$\pdim_QR$ is finite and from $\pdim_{Q/(f)}M<\infty$ as $\pdim_Q Q/(f)$ is finite.
Thus it suffices to establish the equivalence when $\pdim_QM$ is finite.

By \eqref{eq:variety1},  $\pdim_{R}M<\infty$ is equivalent to $\dim V(Q,\bsf,M)=0$.  Since it is defined
by homogeneous equations, $V(Q,\bsf,M)$ is a cone in $\ov I$ with vertex at~$0$.
As $k$ is infinite, $\dim V(Q,\bsf,M)=0$ is equivalent to $V(Q,\bsf,M)=\{0\}$.  The latter holds if and
only if $\pdim_{Q/(f)}M$ is finite for each $f\in I\smallsetminus\fq$, due to \eqref{eq:variety2}.
    \end{proof}

  \begin{chunk}  \textit{Proof of Theorem} \ref{thm:rigidT}.
    \label{ch:PfrigidT}
When $R$ is Tor-persistent, so is $Q$, by Proposition~\ref{prop:descent}.

Suppose that $Q$ is Tor-persistent and $M$ is a finite $R$-module with $\Tor{}{R}{M}{M}$ finite. We
want to prove that $\pdim_{R}M$ is finite. By Proposition \ref{prop:criterion}, it suffices to show that
$\pdim_{Q/(f)}M$ is finite for every $f\in I\smallsetminus\fq I$. Fix one such $f$ and set $Q'=Q/(f)$.
Recall that $f$ is $Q$-regular and the map $R\twoheadleftarrow Q'$ is a deformation.

By \ref{ch:ssTor}, one has  $\Tor i{Q'}MM=0$ for $i\gg 0$. When $N$ is a sufficiently high syzygy module
of $M$ over $Q'$, standard degree-shifting isomorphisms yield
  \begin{equation}
    \label{eq:rigidT2}
\Tor i{Q'}{N}{N}=0
\quad\text{for}\quad
i\ge1\,.
  \end{equation}
Then $\Tor {}{Q}{N}{N}$ is also bounded, again by \ref{ch:ssTor}, so $\pdim_QN$ is finite because $Q$ is
$\tor$-persistent.  The choice of $N$ entails $\depth_{Q'}N=\depth Q'$.  By the Auslander-Buchsbaum
Equality, $N$ has a minimal $Q$-free resolution of the form
  \begin{equation}
    \label{eq:rigidT1}
G=\qquad0\to Q^b\to Q^b\to 0
  \end{equation}

To finish the proof, we show that $\pdim_{Q'}N$ is finite.  Set $F=G\otimes_QG$ and $L=N\otimes_{Q'}N$.
In view of \eqref{eq:rigidT2}, the spectral sequence \eqref{eq:ssTor} for $Q\to Q'$ yields
 \begin{equation}
    \label{eq:rigidT3}
\HH iF\cong \Tor i{Q}{N}{N}\cong
\begin{cases}
 \Tor 0{Q'}{N}{N}\cong L &\text{for }i=0,1\,;
  \\
0 &\text{otherwise}\,.
\end{cases}
  \end{equation}

Let $F'$ be the subcomplex of $F$ with $F'_i=F_i$ for $i\ge 2$, $F'_1=\Ker\partial_1$, and $F'_i=0$
for $i\le0$.  From \eqref{eq:rigidT3} we get quasi-isomorphisms $F'\simeq\shift L$ and $F/F'\simeq L$,
so the exact sequence $0\to F'\to F\to F/F'\to 0$ of complexes yields an exact sequence
 \[
\cdots\to\Tor {i+2}Q{F}k\to \Tor {i+2}Q{L}k\to \Tor {i}Q{L}k\to \Tor {i+1}Q{F}k\to\cdots
  \]
Set $T_i=\Tor iQ{L}k$.  Since $\Tor iQ{F}k\cong F_i\otimes_Qk$ holds for each $i$, the resolution \eqref{eq:rigidT1} and the exact sequence above   yield exact sequences
\begin{align*}
0\to T_{i+2}&\to T_{i}\to 0 \quad\text{for}\quad i\ge2 \\
0\to T_3\to T_1\to k^{b^2}\to T_2&\to T_0\to k^{2b^2}\to T_1\to 0\to k^{b^2}\to T_0\to0
\end{align*}
of vector spaces.  As a result, we get equalities
\begin{align*}
& \beta^Q_i(L)=\beta^Q_{i+2}(L) \quad\text{for $i\ge2$}\\
& \beta^Q_2(L)=\beta^Q_{3}(L)
\end{align*}
We conclude that $\beta^Q_i(L)=\beta^Q_{i+1}(L)$ holds for $i\ge2$.

In view of the hypothesis on $Q$, this implies $\pdim_QL\le1$.

If $G'$ is a minimal $Q'$-free resolution of $N$, then $G'\otimes_{Q'}G'$ is one of $L$, by \eqref{eq:rigidT2}.
This gives the first equality in the next string, and \cite[4.2.5(4)]{Av:barca} yields the inequality:
  \[
2\cx_{Q'}N=\cx_{Q'}L\le\cx_QL+1=1
  \]
As a consequence, we get $\cx_{Q'}N=0$, so $\pdim_{Q'}N$ is finite, as desired.
    \qed
  \end{chunk}

  \section{Tor-friendly rings}
    \label{S:Tor-friendly-rings}

In this section $(R,\fm,k)$ denotes a local ring, and $M$ and $N$ are finite $R$-modules.

Outside of the class of locally complete intersection rings, $\tor$-persistence has been
difficult to verify.  The following stronger property is easier to work with.

\begin{chunk}
  \label{ch:friendly}
We say that $R$ is \emph{$\tor$-friendly} if for every pair $(M,N)$ of finite $R$-modules $\Tor{}RMN$
is bounded  only if $\pdim_RM$ or $\pdim_RN$ is finite.
  \end{chunk}

We record some formal properties of $\tor$-friendliness.  Proofs for the following alternative characterization
and descent result are omitted, as they parallel those of the analogous results for
$\tor$-persistence;  see Proposition~\ref{pr:persistent} and Proposition~\ref{prop:descent}.

  \begin{proposition}
  \label{prop:friendly}
The following conditions on a local ring $R$ are equivalent.
  \begin{enumerate}[\quad\rm(i)]
 \item
The ring $R$ is $\tor$-friendly.
 \item
If $U$ and $V$ are $R$-complexes, such that the $R$-modules $\hh U$ and $\hh V$ are finite and $\Tor{}RUV$ is bounded, then
$U$ or $V$ is perfect.
\item
If $U$ and $V$ are $R$-complexes, such that the $R$-modules $\hh U$ and $\hh V$ are of finite length and
$\Tor{}RUV$ is bounded, then $U$ or $V$ is perfect. \qed
  \end{enumerate}
  \end{proposition}

\begin{proposition}
\label{prop:descent-friendly}
Let $\varphi\colon R\to S$ be a local homomorphism of local rings.

If $\varphi$ is of finite flat dimension and $S$ is $\tor$-friendly, then so is $R$.
   \qed
\end{proposition}

  \begin{remark}
  \label{rem:unfriendly}
When $R$ is singular, $\tor$-friendliness need not ascend from $R$ to $S$, even when $\varphi$ is surjective
with kernel generated by an $R$-regular element: see Proposition~\ref{prop:SegaFriendly}.  This should be
compared to Theorems~\ref{thm:rigidT} and \ref{thm:mainTor}.
    \end{remark}

\begin{proposition}
\label{ch:friendly-complete}
Let $I\subseteq\fm$ be an ideal of $R$, and let $\wh R$ denote the $I$-adic completion of $R$. The ring $R$ is $\tor$-friendly if and only if so is $\wh R$.
  \qed
     \end{proposition}

For the proof of Proposition~\ref{prop:tensor}, we need a general result of independent interest.

\begin{lemma}
  \label{lem:tensor}
Let $P\to Q$ and $P\to Q'$ be homomorphisms of rings.

If $\Tor iPQ{Q'}=0$ for $i\ge1$ and $G$ and $G'$ are bounded below complexes of flat modules
over $Q$ and $Q'$, respectively, then there are $(Q\otimes_PQ')$-linear isomorphisms
  \begin{equation}
    \label{eq:tensor1}
\HH n{G\otimes_{P}G'}\cong\Tor nP{G}{G'}\quad \text{for each}\quad n\ge 1\,.
  \end{equation}
    \end{lemma}

\begin{proof}
Let $F'\to G'$ be a free resolution of $G'$ over $P$. It suffices to prove that the induced morphism
$G\otimes_{P}F' \to G\otimes_{P}G'$ is a quasi-isomorphism; equivalently, that the mapping
cone $C$ of $F'\to G'$ satisfies
  \begin{equation}
    \label{eq:tensor2}
\HH n{G\otimes_{P}C}=0\quad \text{for each}\quad n\in\BZ\,.
  \end{equation}

We proceed by d\'evissage.  Since $F'$ and $G'$ are bounded below, there exists an integer $s$ such
that $C_{j}=0$ for $j\le s$, so there are isomorphisms
\[
\HH n{G\otimes_{P}C}\cong \HH n{ G_{\les n-s+1}\otimes_{P}C}\,.
\]
Thus, replacing $G$ by $G_{\les n-s+1}$ we may assume that $G$ is a bounded complex of flat
$Q$-modules.  Set $b=\min \{j\mid G_{j}\ne 0\}$.  The sequence of complexes of $Q$-modules
\[
0\lra G_{b}\lra G \lra G_{\ges b+1}\lra 0
\]
is split-exact as a sequence of graded modules, so the induced sequence of complexes
\[
0\lra G_{b}\otimes_{P}C\lra G\otimes_{P}C \lra G_{\ges b+1}\otimes_{P}C \lra 0
\]
is exact.  Induction on the number of nonzero terms in $G$ shows that it suffices to prove
\eqref{eq:tensor2} when $G$ is a flat $Q$-module.  We assume that for the rest of the proof.

Let $F\to G$ be a free resolution of $G$ over $P$.  We have standard spectral sequences
\[
{}'E^{2}_{i,j} = \HH j{\HH i{F\otimes_PC_j}} \underset{i}{\Longrightarrow} \HH{i+j}{F\otimes_PC}
 \underset{j}{\Longleftarrow} \HH i{\HH j{F_i\otimes_PC}} = {}''E^{2}_{j,i}
\]
We have $\HH j{F_i\otimes_PC}\cong F_i\otimes_P\HH j{C}=0$ for all $(j,i)$, so ${}''E^{r}_{j,i}$
yields $\HH{i}{F\otimes_PC}=0$.  As $\HH i{F\otimes_PC_j}\cong\Tor iPG{C_j}$
holds for all $(i,j)$, we get a spectral sequence
  \[
E^{2}_{i,j} = \HH j{\Tor iPG{C_j}} \underset{i}{\Longrightarrow} 0
  \quad\text{with}\quad
E^{2}_{0,j} = \HH j{G\otimes_PC}\,.
  \]
The $Q$-module $F'_j$
is free, so for every integer $j$ and each $n\ge 1$ we have isomorphisms
\[
\Tor nPG{C_{j}} = \Tor nPG{F'_{j-1}\oplus G'_{j}} \cong \Tor nPG{G'_{j}}\,.
\]
Therefore it suffices to prove \eqref{eq:tensor1} when $G'$ is a flat $Q'$-module.

By the Govorov-Lazard Theorem, $G$ and $G'$ are filtered colimits of families $\bsG=\{G_{i}\}_{i\in I}$
and $\bsG'=\{G'_{i'}\}_{i'\in I'}$ of finite free modules over $Q$ and $Q'$, respectively, so
  \[
\Tor nPG{G'}
\cong\Tor nP{\underset{i\in I}{\colim}\bsG'}{\underset{i'\in I'}{\colim}\bsG}
\cong\underset{(i,i')\in I\times I'}{\colim}\Tor nP{G_i}{G'_{i'}}\,.
  \]

Formula \eqref{eq:tensor1} evidently holds for $G\cong Q^r$ and $G'\cong Q'^{r'}$.
  \end{proof}

Here is a first application of Proposition~\ref{prop:friendly}.

    \begin{proposition}
  \label{prop:tensor}
Let $Q$ and ${Q'}$ be singular residue rings of a regular local ring $P$.

If $\Tor iPQ{Q'}=0$ for $i\ge1$, then the ring $R=Q\otimes_P{Q'}$ is not $\tor$-friendly.
    \end{proposition}

   \begin{proof}
Since $Q$ and ${Q'}$ are singular, they have finite modules of infinite projective dimension.  Let $L$ and $L'$ be such modules, and $G$ and $G'$ be their minimal free resolutions over $Q$ and ${Q'}$, respectively.  Since $G$ and $G'$ are minimal, so are the complexes of finite free $R$-modules $G\otimes_{P}Q'$ and $Q\otimes_{P}G'$, and hence of infinite projective dimension over $R$.
From Lemma~\ref{lem:tensor}, applied with $G'=Q'$, and the regularity of $P$ it follows that
\[
\HH i{G\otimes_{P}Q'}\cong \Tor iPG{Q'}\cong \Tor iPL{Q'}=0\quad\text{for $i\gg 0$.}
\]
By symmetry, $\hh{Q\otimes_{P}G'}$ is bounded as well. Consider the following chain of isomorphisms, where the first two are standard.
\begin{align*}
\Tor iR{G\otimes_{P}Q'}{Q\otimes_{P}G'}
	& \cong \HH i{(G\otimes_{P}Q') \otimes_{(Q\otimes_{P}Q')}(Q\otimes_{P}G')} \\
	&  \cong \HH i{G\otimes_{P}G'}\\
	& \cong \Tor iPL{L'} \\
	&\cong 0 \quad \text{for $i\gg 0$.}
\end{align*}
The third isomorphism holds by Lemma~\ref{lem:tensor}, and the last one because $P$ is regular.  It remains to apply Proposition~\ref{prop:friendly} to conclude that $R$ is not $\tor$-friendly.
   \end{proof}

In part (2) of the next result, the ``if'' direction is a theorem of Huneke and Wiegand \cite[1.9]{HW}; a different
proof was given by Miller \cite[1.1]{Mi}.  The converse assertion is due to \c Sega~\cite[4.2]{Se1}; the proof given below is different.

   \begin{proposition}
  \label{prop:SegaFriendly}
Assume $\wh R$ is isomorphic to $Q/(f)$, where $(Q,\fq,k)$ is a local ring and $f$ is a $Q$-regular element.
  \begin{enumerate}[\rm(1)]
    \item
When $f$ does not lie in $\fq^2$, the ring $R$ is $\tor$-friendly (respectively, $\tor$-persistent) if and only if $Q$ is.
    \item
When $f$ lies in $\fq^2$, the ring $R$ is $\tor$-friendly if and only if $Q$ is regular.
    \end{enumerate}
      \end{proposition}

 \begin{proof}
In view of Proposition \ref{ch:friendly-complete}, it suffices to prove the statement for $\wh R$, so we may assume that
$R$ is complete and is equal to $Q/(f)$.

(1)  By Proposition \ref{prop:descent-friendly}, we have to show that if $Q$ is $\tor$-friendly and $M$ and $N$ are finite
$R$-modules with $\Tor{}RMN$ bounded, then $\pdim_RM$ or $\pdim_RN$ is finite.  From \ref{ch:ssTor} we know that
$\Tor{}QMN$ is bounded, so $\pdim_QM$ or $\pdim_QN$ is finite.  As $f\notin\fq^2$ holds, a classical result due to Nagata shows that
$\pdim_RM$ or $\pdim_RN$ is finite;  see \cite[2.2.3]{Av:barca}.

The same argument also settles the case of $\tor$-persistence.

(2)  By Proposition \ref{ch:friendly-complete} we  may assume $Q$ is complete. Choose a minimal Cohen presentation $Q\cong P/J$ with a regular local ring $(P,\fp,k)$ and $J\subseteq\fp^2$.  Choose $g$ in $P$ that maps to $f$; as $f$ is in $\fq^2$ we have $g\in\fp^2$, so the ring $Q'=P/(g)$  is singular.  In addition, we have $R\cong Q\otimes_PQ'$, and $\Tor iPQ{Q'}=0$ for $i\ge1$, as $g$ is $Q$-regular.

If $Q$ is singular, then Proposition \ref{prop:tensor} shows that $R$ is not $\tor$-friendly.

If $Q$ is regular, then $J=0$, so $R=Q'$.  This is a hypersurface ring, and Huneke and Wiegand \cite[1.9]{HW} have shown
that hypersurface rings are $\tor$-friendly.
  \end{proof}

  \section{Recognizing friendliness}

In this section $(R,\fm,k)$ denotes a local ring, and $M$ and $N$ are finite $R$-modules.
We collect various sufficient conditions for $R$ to be $\tor$-friendly.  All of them are needed
in Section \ref{S:Homologically-persistent-rings}, for the proof of the main theorem of the paper.

We start by describing those results that bring in the largest haul.

  \begin{chunk}
    \label{ch:Sega1}
The formal power series with non-negative integer coefficients
  \[
P^R_M(z)=\sum_{n\ges0}\beta^R_n(M)z^n
  \]
is known as the \emph{Poincar\'e series} of $M$ over $R$.

A \emph{common denominator} for Poincar\'e series over $R$ is a polynomial $d(z)\in\BZ[z]$,
such that $d(z)P^R_N(z)\in\BZ[z]$ holds for every finite $R$-module~$N$.

Following \cite{Se1} we say that a factorization $d(z) = p(z)q(z)r(z)$ in $\BZ[z]$ is  \emph{good} if the following conditions hold:
\begin{enumerate}[\quad\rm(1)]
\item
 $p(z) = 1$ or $p(z)$ is  irreducible;
\item
$q(z)$ has non-negative coefficients;
\item
$r(z) = 1$ or $r(z)$ is irreducible and none of its complex roots of minimal absolute value is a positive real number.
\end{enumerate}

\medskip

If the Poincar\'e series over $R$ admit a common denominator $d(z)$ that has a good factorization,
then $R$ is $\tor$-friendly:  This is proved by \c Sega \cite[1.4 and 1.5]{Se1}.
  \end{chunk}

\begin{chunk}
\label{ch:trivial}
A graded ring $A=\bigoplus_{i\in\BZ}A_i$ is said to be \emph{graded-commutative} if the identities
$aa' = (-1)^{ii'}a'a$ and $a^{2}=0$ when $i$ is odd hold for all $a\in A_i$ and $a'\in A_{i'}$.

When $W=\bigoplus_{j\in\BZ}W_j$ is a graded $A$-module, the \emph{trivial extension} $A\ltimes W$
is the graded ring with underlying graded additive group $A\oplus W$ and product
\[
(a,w)(a',w')=(aa',aw'+(-1)^{ji'}a'w)\quad\text{for $a'\in A_{i'}$ and $w\in W_j$.}
\]
Note that $A\ltimes W$ is also graded-commutative. We identify $A$ and $W$ with their images in
$A\ltimes W$; note that $A$ is a subring and $W$ is an ideal with $W^2=0$.

The following result is proved in \cite[5.3]{AINS} specifically for use in the present paper:

\medskip

The ring $R$ is $\tor$-friendly if some Cohen presentation $\wh R\cong P/I$ satisfies
  \begin{enumerate}[\quad\rm(a)]
   \item
a minimal free resolution of $\wh R$ over $P$ has a structure of DG algebra; and
   \item
the $k$-algebra $B=\Tor{}P{\wh R}k$ is isomorphic to a \emph{trivial extension} $A\ltimes W$ of a
graded $k$-algebra $A$ by a graded $A$-module $W\ne0$ with $A_{\ges 1}\cdot W=0$.
  \end{enumerate}
It is easy to see that these conditions do not depend on the choice of presentation.
  \end{chunk}

A precursor of \cite[5.3]{AINS} was proved by Nasseh and Yoshino \cite[3.1]{NY}.
We give a short, independent proof of that result as part (2) of the next proposition.

   \begin{proposition}
  \label{prop:trivialLoc}
When $\pdim_RM$ is infinite the following assertions hold.
\begin{enumerate}[\rm(1)]
    \item
If $k$ is a direct summand of $\syz [R]hN$, then $\Tor{i}RMN\ne0$ for $i\ge h$.
    \item
If $(0:\fm)\nsubseteq\fm^2$, then $k$ is a direct summand of $\syz [R]iM$ for each $i\ge2$, and hence $R$ is $\tor$-friendly.
  \end{enumerate}
  \end{proposition}

  \begin{proof}  (1)  We have $\Tor{i}R{M}{N}\cong\Tor{i-h}R{M}{\syz [R]hN}\supseteq\Tor{h}R{M}{k}\ne0$.

(2)  We will show that if an $R$-linear map $\delta\colon L\to L'$ of finite free modules has $\Ker\delta\subseteq \fm L$
and $\Image\delta\subseteq \fm L'$, then $\Ker\delta$ has a direct summand isomorphic to $L/\fm L$.

Set $D=\Ker\delta$ and pick $x\in (0:\fm)\smallsetminus\fm^2$.  Then we have $\delta(xL)=x\delta(L)=0$,
and hence $D\supseteq xL\cong L/\fm L$.  The composed map $\gamma\colon xL\to D/\fm D$ satisfies
  \[
\Ker\gamma=\fm D\cap xL\subseteq \fm^2L\cap xL=(\fm^2\cap xR)L=0
  \]
Choose in $D$ a submodule $D'$ containing $\fm D$, so that $D/\fm D=(D'/\fm D)\oplus\gamma(xL)$.
It is not hard to verify that $D=D'\oplus xL$.
  \end{proof}

   \begin{chunk}
     \label{chunk:hilbert}
Recall that the formal power series with non-negative integer coefficients
  \[
H_M(z)=\sum_{n\ges0}\rank_k(\fm^nM/\fm^{n+1}M)z^n
  \]
is called the \emph{Hilbert series} of $M$.  By the Hilbert-Serre Theorem, it represents a
rational function $h_M(z)/(1-z)^{\dim M}$ with $h_M(z)\in\BZ[z]$ and $h_M(1)\ge1$.
   \end{chunk}

The result below and its proof are in the spirit of those in \cite{HSV}.

    \begin{theorem}
  \label{thm:short}
Let $(R,\fm,k)$ be a local ring with $H_R(z)=1+ez+sz^2$.

If $s=0$, or if $e^2-4s$ is not the square of an integer, then $R$ is $\tor$-friendly.

If $s=0$, or if $e^2-4s$ is not zero, then $R$ is $\tor$-persistent.
      \end{theorem}

Indeed, this is a consequence of the first assertion of the next result.

    \begin{proposition}
  \label{prop:short}
Let $(R,\fm,k)$ be a local ring with $H_R(z)=1+ez+sz^2$.

If $M$ and $N$ are finite $R$-modules of infinite projective dimension and $\Tor {}RMN$ is bounded,
then there are positive integers $u$ and $v$ such that
  \begin{equation}
    \label{eq:short1}
H_R(z)=(1+uz)(1+vz)\,.
  \end{equation}

Furthermore, there exist finite $R$-modules $M'$ and $N'$ of infinite projective dimension,  such that
$\Tor {}R{M'}{N'}$ is bounded and the following equalities hold:
  \begin{equation}
    \label{eq:short2}
H_{M'}(z)=m(1+uz)\ne0
  \quad{and}\quad
H_{N'}(z)=n(1+vz)\ne0\,.
  \end{equation}
In case $N=M$, the preceding conclusions hold with $N'=M'$ and $u=v$.
    \end{proposition}

    \begin{proof}
Following \cite[3.1]{Le1}, we say that $N$ is \emph{exceptional} if $\fm^2N=0$ and $\syz [R]jN$ has no
direct summand isomorphic to $k$, for $j\ge1$.

Fix $h\ge1$ so that $\Tor{i}RMN=0$ holds for $i>h$.  For $j\ge h$ we then have $\Tor{i}RM{\syz [R]jN}=0$ for  $i\ge 1$, and also
$\fm^2\syz [R]jN=0$ because $\syz [R]jN$ is a syzygy module.  It follows from Proposition \ref{prop:trivialLoc}(1) that $\fm\syz [R]jN\ne0$, and hence  $\syz [R]jN$ is exceptional. By symmetry,  $\syz [R]iM$ has the corresponding properties for $i\gg0$.

Thus, suitable syzygy modules $M'$ of $M$ and $N'$ for $M$ satisfy the conditions:
   \begin{align}
    \label{eq:short3}
&M'  \text{ and } N' \text{ are exceptional, not free, }\Tor{h}R{M'}{N'}=0 \text{ for }h\ge1\,.
   \\
    \label{eq:short4}
&H_{M'}(z)=m_0+m_1z
  \quad\text{and}\quad
H_{N'}(z)=n_0+n_1z
  \quad\text{with}\quad
m_i,n_j\ge1\,.
  \end{align}
Note that if $N=M$, then we may choose $N'=M'$.

As $M'$ is exceptional, by Lescot \cite[3.4 and 3.6]{Le1} we have
  \begin{equation}
    \label{eq:short5}
P^R_{M'}(z)=\frac{H_{M'}(-z)}{H_R(-z)}
  \end{equation}

The exact sequence $0\to\fm {N'}\to {N'}\to {N'}/\fm {N'}\to0$ induces an exact sequence
 \[
\cdots\to\Tor {i}R{M'}{N'}\to\Tor {i}R{M'}k^{n_0}\to \Tor{i-1}R{M'}k^{n_1}\to \Tor{i-1}R{M'}{N'}\to\cdots
  \]
Since $\Tor iR{M'}{N'}=0$ for $i\ge1$, we obtain equalities
   \stepcounter{equation}
   \begin{equation}
    \label{eq:short6}
        \tag{\ref{prop:short}.6.${i}$}
{n_0}\beta^R_{i}({M'})={n_1}\beta^R_{i-1}({M'})
  \quad\text{for}\quad
i\ge2
  \end{equation}
and an exact sequence
  \[
0\to\Tor {1}R{M'}k^{n_0}\xra{\delta}\Tor {0}R{M'}k^{n_1}\to {M'}\otimes_R{N'}\to \Tor 0R{M'}k^{n_0}\to0
  \]
Setting $r=\rank_k\Coker(\delta)$ and $l=\length_R({M'}\otimes_R{N'})$, we further get
  \begin{align}
    \tag{\ref{prop:short}.6.1}
{n_0}\beta^R_{1}({M'})&={n_1}\beta^R_{0}({M'})-r
  \\
    \tag{\ref{prop:short}.6.0}
{n_0}\beta^R_{0}({M'})&=l-r
  \end{align}
Multiplying \eqref{eq:short6} by $z^i$, for $i\ge0$, and adding the resulting equalities yields
  \[
{n_0}P^R_{M'}(z)={n_1}zP^R_{M'}(z)-rz+(l-r)
  \]
In view of formulas \eqref{eq:short5} and \eqref{eq:short4}, the preceding equation gives
   \[
((l-r)-rz)(1-ez+sz^2)=({m_0}-{m_1}z)({n_0}-{n_1}z)
  \]

Proposition \ref{prop:trivialLoc}(2) gives $\fm^2\ne0$, so we have $s\ne0$.  Matching degrees and
constant terms in the last equality yields $r=0$ and $0\ne l={m_0n_0}$, so we obtain
   \begin{equation}
    \label{eq:short7}
{m_0n_0}(1-ez+sz^2)=({m_0}-{m_1}z)({n_0}-{n_1}z)
  \end{equation}
Substituting $1/z$ for $z$, then multiplying both sides by $z^2/m_0n_0$ yields
  \[
z^2-ez+s
=\left(z-\frac{m_1}{m_0}\right)\left(z-\frac{n_1}{n_0}\right)
  \]
Thus, $u=m_1/m_0$ and $v=n_1/n_0$ are integers, and $m_1$ and $n_1$ are non-zero.

Now formulas \eqref{eq:short7} and \eqref{eq:short4} turn into \eqref{eq:short1} and \eqref{eq:short2}, respectively.
    \end{proof}

  \section{Homologically persistent rings}
     \label{S:Homologically-persistent-rings}

In this section $(R,\fm,k)$ denotes a local ring and $\mult R$ its \emph{multiplicity}; that is,
the positive integer $h_R(1)$ defined by the Hilbert series $H_R(t)$ in \ref{chunk:hilbert}.
Other notions appearing in the hypotheses of the next result are defined in \ref{ch:bounds},
\ref{ch:codepth1}, and \ref{ch:linked}.

  \begin{theorem}
    \label{thm:mainTor}
Assume that there exist a local homomorphism $R\to R'$ of finite flat dimension and a
deformation $R'\twoheadleftarrow Q$, where $Q$ satisfies one of the conditions
  \begin{enumerate}[\quad\rm(a)]
 \item
$\edim Q-\depth Q\le3$.
 \item
$Q$ is Gorenstein and $\edim Q-\depth Q=4$.
 \item
$Q$ is Cohen-Macaulay, almost complete intersection, $\edim Q-\depth Q=4$, and $\frac12\in Q$.
 \item
$Q$ is complete intersection.
 \item
$Q$ is one link from a complete intersection.
 \item
$Q$ is two links from a complete intersection and is Gorenstein.
 \item
$Q$ is Golod.
  \item
$Q$ is Cohen-Macaulay and $\mult Q\le7$.
  \end{enumerate}

The ring $R$ is $\tor$-persistent.
 \end{theorem}

The proof of the theorem takes up the balance of this section.  It involves several steps, put together in \ref{ch:PfmainTor}.  Some special cases are known from earlier work.

  \begin{remark}
    \label{rem:known}
The conclusion of the theorem is known when $R=R'=Q$  and one of the following conditions holds: $R$ is
complete intersection (by \cite{AB}; see Example \ref{ex:lci-unbounded}); $\edim R-\depth R\le1$
(for then $R$ is complete intersection); $R$ is Golod and $\edim R-\depth R\ge2$ (by Jorgensen,
\cite[3.1]{Jo}); $\edim R-\depth R=2$ (for then $R$ is either complete intersection or Golod; see \cite[5.3.4]{Av:barca});
$R$ is Gorenstein with $\edim R-\depth R=3,4$ and $\wh R$ has no embedded deformation (by
\c Sega, \cite[2.3]{Se1});  $R$ is Gorenstein of minimal multiplicity (by \cite[1.8]{Se1}, and also \cite[3.7]{HJ}); $R$ is Cohen-Macaulay with small multiplicity or $\edim R - \dim R \leq 3$ (by \cite{LM}).

When $Q$ is Cohen-Macaulay of minimal multiplicity, the ring $R$ is Tor-persistent by \cite[4.7,6.5]{NT}; see also \cite[1.3]{GP}. This result can also be deduced from Theorem~\ref{thm:rigidT}: We can assume $\edim Q-\depth Q\ge 2$ (else $Q$ is a hypersurface), and then, as $Q$ has minimal multiplicity,  the Betti numbers of finitely generated modules are either eventually zero or grow exponentially; see \cite[4.2.6]{Av:barca}.

In all the remaining cases the result is new, even when $R=R'=Q$.
  \end{remark}

  \begin{chunk}
    \label{ch:bounds}
Let $\wh Q\cong P/J$ be a  minimal Cohen presentation and set
\[
e=\edim P\quad\text{and}\quad c=\beta^P_1(\wh Q)\,.
\]
For the residue field $k'$ of $Q$ there are coefficientwise inequalities of power series
  \begin{equation}
    \label{eq:bounds}
\frac{(1+z)^e}{(1-z^2)^c}
\preccurlyeq P^Q_{k'}(z)
\preccurlyeq\frac{(1+z)^e}{1+z-zP^P_{\wh Q}(z)}
  \end{equation}

Equality holds on the left in \eqref{eq:bounds} if and only if $Q$ is \emph{complete intersection}.
This is equivalent to $\edim Q-\dim Q=c$ and implies that $Q$ is Cohen-Macaulay.

The ring $Q$ is said to be \emph{almost complete intersection} if $\edim Q-\dim Q=c-1$.
    \end{chunk}

  \begin{chunk}
    \label{ch:codepth1}
When equality holds on the right in \eqref{eq:bounds}, the ring $Q$ is said to be \emph{Golod}.

If $\edim Q-\depth Q\le 1$, then $c=\edim Q-\depth Q$ holds, so the upper and lower bounds
in \eqref{eq:bounds} coincide; thus $Q$ is both complete intersection and Golod.

When $\edim Q-\depth Q\ge 2$ and $Q$ is Golod, it has no embedded deformation.

Indeed, assume $Q\cong P'/(g)$ for a local ring $P'$ with $\edim P'=\edim Q$ and a $P'$-regular element $g$.
Let $P\to \wh{P'}$ be a minimal Cohen presentation. Composing it with the induced map $\wh{P'}\to\wh Q$
yields a minimal Cohen presentation of $\wh Q$.  There is then an isomorphism
$\Tor{}P{\wh Q}k\cong\Tor{}{P}{\wh {P'}}{k'}\otimes {k'}\langle h\rangle$ as graded $k'$-algebras, where
${k'}\langle h\rangle$ is an exterior algebra on a generator of degree $1$; see \cite[3.3]{Av:msri}.  We get
  \[
\rank_{k'}\Tor{1}P{\wh Q}{k'}^2\ge\rank_k(\Tor{1}P{\wh {P'}}{k'}\otimes_{k'}(k'h))=c-1\ge1
  \]
This is impossible, as Golod's Theorem \cite{Go} shows that all Massey products in
$\Tor{}P{\wh Q}{k'}$ are equal to $0$; in particular, $\Tor{1}P{\wh Q}{k'}^2=0$ holds.
    \end{chunk}

  \begin{chunk}
    \label{ch:linked}
Ideals $J$ and $J'$ in a regular local ring $P$ are said to be \emph{linked} if there exists a $P$-regular sequence $\bsg$ in $J\cap J'$ such that $J=(P\bsg:J')$ and $J'=(P\bsg:J)$.

We say that $Q$ is $s$ \emph{links from a complete intersection} if in some minimal Cohen presentation
$\wh Q\cong P/J$ there is a sequence $J_1,\dots,J_s$ of ideals with $J=J_1$, $J_i$ linked to
$J_{i+1}$ for $i=1,\dots,s-1$, and $J_s$ generated by a $P$-regular sequence; see \cite[3.3]{AKM}.
    \end{chunk}

We finish the preliminary discussion with a classical construction.

  \begin{chunk}
    \label{ch:inflation}
If $k\to\ov k$ is a field extension, then there is a flat ring homomorphism $R\to\ov R$, where
$\ov R$ is a local ring with maximal ideal $\fm\ov R$ and residue field $\ov k$, and the induced
map $R/\fm\to\ov R/\fm\ov R$ is the given field extension; see \cite[IX, App., Theorem 1, Cor.]{Bo}.
    \end{chunk}

The next three items contain the crucial reductions for the proof of the theorem.

  \begin{lemma}
    \label{lem:nodef}
Under the hypothesis of Theorem~\emph{\ref{thm:mainTor}} the ring $Q$ may be chosen to be
complete, with algebraically closed residue field, and with no embedded deformation.
  \end{lemma}

  \begin{proof}
Let $R\to R'\twoheadleftarrow Q$ be the maps given in Theorem \ref{thm:mainTor}.  Let
$\wh Q\cong P/J$ be a minimal Cohen presentation, which in cases (e) and (f) satisfies the
properties~in~\ref{ch:linked}.

Pick a flat homomorphism of local rings $P\to P'$ such that $P'/\fp P'$ is an algebraic closure
$l$ of the residue field of $P$; let $\wt P$ be the completion of $P'$; see \ref{ch:inflation}.  The ring
$\wt Q=\wh Q\otimes_P\wt P$ is local and complete with residue field $l$.  As $\wt P$ is regular, we get
a Cohen presentation $\wt Q\cong\wt P/J\wt P$, and it is minimal due to the equalities
   \[
\edim \wt P=\edim P=\edim Q=\edim\wt Q\,.
  \]
Standard results yield $\depth \wt Q=\depth Q$ and $\mult \wt Q=\mult Q$, and show that $\wt Q$ is
Gorenstein, Cohen-Macaulay, (almost) complete intersection, or linked in~$s$
steps to a complete intersection whenever $Q$ has the corresponding property.  The equalities
$P^{\wt P}_{\wt Q}(z)=P^P_{\wh Q}(z)$ and $P^{\wt Q}_{l}(z)=P^Q_{k}(z)$ show that $\wt Q$ is Golod
if $Q$ is; see~\ref{ch:bounds}.

In cases (a)--(g) the surjective map $\wt Q\gets \wt P$ can be factored as $\wt Q\twoheadleftarrow Q'\gets\wt P$, where
$\wt Q\twoheadleftarrow Q'$ is a deformation, $Q'$ satisfies one of properties (a) through (g), and
$Q'$ has no embedded deformation.  Indeed, the desired factorizations are provided by \cite[3.1]{Av:msri}
in cases (a), (b), (e), and (f), and by \cite[1.2 and 1.3]{KP} in case~(c).  In case~(d), take $Q'=\wt P$.
By \ref{ch:codepth1}, this covers case (g) when $\edim Q-\depth Q\le1$, while $Q'=\wt Q$ works when
$\edim Q-\depth Q\ge2$.

For case (h), choose an embedded deformation $\wt Q\twoheadleftarrow Q'$ with $\edim\wt Q-\depth Q'$ maximal and recall that $\mult Q\le\mult \wt Q$ holds; see \cite[VIII, \S7, Prop.~4]{Bo}.

Set $R''=\wh{R'}\otimes_P\wt P$.  We have a diagram of local homomorphisms of local rings
  \[
R\to R'\to\wh{R'}\to R''=\wh{R'}\otimes_P\wt P\twoheadleftarrow\wh Q\otimes_P\wt P=\wt Q\twoheadleftarrow Q'
  \]
The composed maps $R\to R''\twoheadleftarrow Q'$ have the desired properties.
  \end{proof}

  \begin{chunk}
    \label{ch:formal}
Let $Q$ be a complete local ring with algebraically closed residue field, $Q\cong P/J$ a minimal
Cohen presentation, and $G$ a minimal free resolution of $ Q$ over $P$.

If $Q$ satisfies one of conditions (a) through (f) in Theorem~\ref{thm:mainTor}, then $G$ admits a structure
of graded-commutative DG $Q$-algebra:  See Buchsbaum and Eisenbud \cite[1.3]{BE} in case (a);  Kustin and
Miller \cite[4.3]{KM1} (when $Q$ contains $\frac12$) and Kustin \cite[Theorem]{Ku1} (when $Q$ contains
$\frac13$) in case (b); Kustin \cite[3.13]{Ku3} in case~(c);  the Koszul complex on a minimal generating set of $J$ in case (d);
Avramov, Kustin, and Miller \cite[4.1]{AKM} in case~(e); Kustin and Miller \cite[1.6]{KM2} in case (f).
  \end{chunk}

 \begin{lemma}
    \label{lem:friendly}
Let $Q$ be a complete local ring with algebraically closed residue field.

If $Q$ has no embedded deformation in the sense of \emph{\ref{ch:def}}, and satisfies one of conditions
\emph{(a)} through \emph{(h)} in Theorem \emph{\ref{thm:mainTor}}, then $Q$ is $\tor$-friendly.
  \end{lemma}

  \begin{proof}
Cases are tackled one at a time.

\medskip

(b)  This is contained in \cite[2.3]{Se1}.

\medskip

(d)  Since $Q$ is a complete local ring with no embedded deformations, $Q$ is complete intersection
precisely when it is regular. This fact will be used again, later in the proof. It settles (d), for regular rings
are evidently $\tor$-friendly.

\medskip

(g)  This is contained in \cite[3.1]{Jo}; see also \cite[Proposition 5.2]{AINS}.

\medskip

The argument in cases (a), (c), (e), and (f) have a similar structure, which we describe next.
Choosing a minimal Cohen presentation $Q\cong P/I$, form the graded $l$-algebras
$B=\Tor{}PQl$ where $l$ is the residue field of $Q$.  Explicit multiplication tables for these algebras
are available in the literature, but sometimes require additional hypotheses on $l$.  These are
covered by our hypothesis that $l$ is algebraically closed, so we may invoke decompositions
$B=A\ltimes W$.

When $W\ne0$ the ring $Q$ is $\tor$-friendly by \ref{ch:trivial},
which applies because minimal DG algebra resolutions exist, as recalled in~\ref{ch:formal}.

When $W=0$ we utilize the fact that the Poincar\'e series of finite $Q$-modules admit a common
denominator, and that explicit denominators are known in all cases under consideration.
We find good factorizations for those polynomials and apply \c Sega's criterion to conclude
that $Q$ is $\tor$-friendly; see~\ref{ch:Sega1}.

\medskip

(a)  If $\edim Q-\depth Q=0$, then $Q$ is regular and so $\tor$-friendly, so we may assume $\edim Q - \depth Q\ge 1$.
We may further assume $\edim Q-\depth Q\ge2$, because $Q$ cannot to be complete intersection.

If $\edim Q-\depth Q=2$ holds, then $Q$ has to be Golod, and so is covered by (g).

If $\edim Q-\depth Q=3$, then $A$ belongs to one of three types:  See
\cite[2.1]{AKM}, from where we take the names of the types and the decomposition $B=A\ltimes W$.

Type $\mathbf{CI}$ consists of complete intersections, so it does not occur here.

In type $\mathbf{TE}$ the relations $\rank_lA_3=0<1\le\rank_lB_3$ imply $W_1\ne0$.

In type $\mathbf{B}$ the relations $\rank_lA_1=2<3\le\rank_lB_1$ imply $W_1\ne0$.

In type $\mathbf{G}(r)$ the algebra $A$ has Poincar\'e duality in degree $3$.  If $W=0$, then
$Q$ is Gorenstein by \cite[Theorem]{AG}, and so it is $\tor$-friendly by \cite[2.3]{Se1}.

In type $\mathbf{H}(p,q)$ the algebra $A$ is a free module over a subalgebra isomorphic to $k\ltimes\shift k$.
As $Q$ has no embedded deformation, this implies $W\ne0$; see \cite[3.4]{Av:vpd}.

\medskip

(c)  The multiplication table of $A$ belongs to one of $12$ types, which are determined in \cite[1.1]{KP}
and displayed in \cite[Table 1, p.\,275]{KP}; we use the names assigned there and set
$t=\rank_lB_3$.  The ring $Q$ does not belong to type $\mathbf{B}[t]$, $\mathbf{C}[t]$, or
$\mathbf{C}^{\star}$ because \cite[1.2]{KP} shows that rings of these types admit embedded deformations.

In the remaining cases \cite[Table 3, p.\,281]{KP} yields $W\ne0$, unless $Q$ is of type
$\mathbf{D}^{(2)}$ with $t=1$, or $\mathbf{E}^{(3)}$ with $t=1$, or $\mathbf{F}^{(4)}$ with $t=1$, or
$\mathbf{F}^{\star}$ with $t=2$.  From the proof of \cite[4.2]{KP} one sees that for the first three types
a common denominator $d(z)$ can be chosen from \cite[Table 2, p.\,280]{KP}.  A common denominator for
the fourth type is given at the bottom of \cite[p.\,289]{KP}.  Here are the explicit values:
  \begin{alignat*}{2}
  \tag*{$\mathbf{D}^{(2)}$}
(t&=1)\colon\quad
&d(z)&=(1-2z-2z^2+5z^3-2z^4-2z^5+z^6)(1+z)^2
\\
  \tag*{$\mathbf{E}^{(3)}$}
(t&=1)\colon\quad
&d(z)&=(1-2z-2z^2+5z^3-2z^4-4z^5+z^6+z^7)(1+z)^2
\\
  \tag*{$\mathbf{F}^{(4)}$}
(t&=1)\colon\quad
&d(z)&=(1-2z-2z^2+5z^3-2z^4-7z^5+z^6+4z^7-z^9)(1+z)^2
\\
  \tag*{$\mathbf{F}^{\star}$}
(t&=2)\colon\quad
&d(z)&=(1-2z-2z^2+5z^3-3z^4-9z^5+z^6+2z^7-z^8)(1+z)^2\\
     &&&=(1 - 5z + 10z^2 - 11z^3 + 5z^4 - z^5)(1+z)^5
  \end{alignat*}

Let $p(z)$ denote the first factor in the final form of an expression for $d(z)$.
Verifications with \emph{\textsf{Mathematica}} show that $p(z)$ is irreducible over $\BQ$.  Thus, we
get good factorizations $d(z)=p(z)q(z)r(z)$ with $q(z)$ a power of $1+z$ and $r(z)=1$.

\medskip

When discussing the next two cases we set $m=\grade_P Q$.

\medskip

(e)  As $Q$ is not complete intersection, $m\ge2$ holds.  The DG algebra structure on a
minimal free resolution of $Q$ over $P$, constructed in the proof of \cite[4.4]{AKM}, gives
$B=A\ltimes W$ with $\rank_{k}W_i=\binom m{i-1}$ for $i=1,\dots m-1$; in particular, $W\ne0$.

\medskip

(f)  We have $m\ge3$ because otherwise $Q$ is a complete intersection.  It is proved in \cite[6.3 and 5.18]{AKM}
that there exists a common denominator $d(z)$ for Poincar\'e series over $Q$, and it satisfies
$d(z)P^Q_k(z)=(1+z)^e$ with $e=\edim Q$.  By using the expression for $P^Q_k(z)$,
obtained in \cite[Theorem 3]{HS} (see also \cite[2.4]{JKM}), we get
  \[
d(z)=((1-z)^m-z)(1+z)^{e-m}\,.
  \]
Set $f(z)=(1-z)^m-z$ and $q(z)=(1+z)^{e-m}$.

The substitution $y=z-1$ turns $f(z)$ into $g(y)=(-1)^my^m-y-1$.  The polynomial $(-1)^mg(y)$ is factored by Selmer
\cite[Theorem 1]{Sl}; Ljunggren \cite[Theorem 3]{Lj} greatly simplified the proof.  The result is that if $m\not\equiv 5\pmod 6$,
then $(-1)^mg(y)$ is irreducible; else, $(-1)^mg(y)=(y^2+y+1)h(y)$, and $h(y)$ is irreducible.

If $m\not\equiv 5\pmod 6$, set $p(z)=f(z)$ and $r(z)=1$; else, set $p(z)=(-1)^mh(z-1)$ and
$r(z)=z^2-z+1$.  In either case, $d(z)=p(z)q(z)r(z)$ is a good factorization.

\medskip

(h)  Since $k$ is infinite, there is a $Q$-regular sequence $\bsg$ that is linearly independent modulo
$\fq^2$ and such that the ring $S=Q/(\bsg)$
has $\length S=\mult Q$; thus, $\length S\le7$ holds.  By Proposition \ref{prop:SegaFriendly}(1),
it suffices to show that $S$ is $\tor$-friendly. Set $\fn=\fq S$, $e=\edim S$ and $s=\rank_k(\fn^2/\fn^3)$.

If $e\le3$, then $S$ is $\tor$-friendly by the already proved case (a) of the theorem, so we may assume
$e\ge4$.  When $\fn^3=0$ we have $s=\length S-e-1\le2$; this implies $s=0$ or $e^2-4s$ is not the square of an integer, so $S$ is
$\tor$-friendly by Theorem \ref{thm:short}.   When $\fn^3\ne0$ the only possibility is $H_S(z)=1+4z+z^2+z^3$.
If $S$ is Gorenstein, then it is $\tor$-friendly, by case (b).  Else, $(0:\fn)\nsubseteq\fn^2$ holds, so $S$ is
$\tor$-friendly by Proposition \ref{prop:trivialLoc}(2).
  \end{proof}

  \begin{chunk}\textit{Proof of Theorem} \ref{thm:mainTor}.
    \label{ch:PfmainTor}
In view of Proposition \ref{prop:descent}, it suffices to show that $R'$ is $\tor$-persistent.  Thus, for
the rest of the proof we assume $R'=R$.

Due to Lemma \ref{lem:nodef}, we may assume that $Q$ is complete, has no embedded
deformation, and its residue field is algebraically closed.  Now Lemma \ref{lem:friendly} shows
that $Q$ is $\tor$-friendly.  In view of Theorem~\ref{thm:rigidT}, the desired assertion will
follow once we show that there is no finite $Q$-module $L$ with $\cx_QL=1$.

In case (d) this is evident:  Complete intersections with no embedded deformations are regular,
so $\cx_QL=0$.  When $Q$ is Golod and not complete intersection, Lescot \cite[6.5]{Le2}
proved that $\pdim_QM=\infty$ implies $\beta^Q_n(M)<\beta^Q_{n+1}(M)$ for $n\gg0$,
and this settles case (g); see also \cite[5.3.3]{Av:barca}.  In case (h) the same conclusion was
obtained by Gasharov and Peeva \cite[1.1]{GP}, provided $\edim Q-\depth Q\ge4$; when
$\edim Q-\depth Q\le3$ holds case (a) applies; it is treated below.

In the remaining cases we argue by contradiction.  Assume that there exists a finite $Q$-module $L$
with $\cx_QL=1$.  Replacing it with a sufficiently high syzygy module, we may further obtain
$\depth_QL=\depth Q$; see \cite[1.2.8]{Av:barca}.  From \cite[1.6.II]{Av:msri} we see that in cases (a),
(b), (e), or (f) the virtual projective dimension of $L$ (defined in \cite[3.3]{Av:vpd}) is equal to $1$; by
\cite[5.2(3)]{KP} the same conclusion holds in case~(c) as well.  Since $Q$ is complete with
infinite residue field, \cite[3.4(c)]{Av:vpd} yields a deformation $Q\twoheadleftarrow P$ with
$\edim P=\edim Q$ and $\pdim_{P}L=1$.  Since $Q$ has no embedded deformation we must have
$P=Q$.  This gives $\pdim_{Q}L=1$, and hence $\cx_QL=0$.  We have produced the desired contradiction,
so the proof of the theorem is complete.
  \qed
  \end{chunk}

\begin{chunk}
\label{ch:examples}
The preceding result gives a direction to a search for rings that may not be Tor-persistent. Naturally, the first example to check is the artinian ring  in \cite[3.4]{GaP} (reproduced in \cite[5.1.4]{Av:barca}), which has embedding dimension 4 and length 8. However its  Hilbert series is $1+4t+3t^2$  and so Theorem~\ref{thm:short}(2) yields that the ring is Tor-persistent. In fact, any local ring $(R,\fm,k)$ with $\fm^3=0$ is Tor-persistent; this is proved in \cite{LW}.
\end{chunk}


\section{Cohomological persistence}
 \label{sec:Vanishing of Ext}
In this section we explore the cohomology analogue of Tor-persistence.  Throughout, $(R,\fm,k)$ will be a local ring.

\begin{chunk}
  \label{ch:epersistent}
We say that $R$ is \emph{cohomologically persistent}, or \emph{$\ext$-persistent}, if every $R$-complex $M$ for which $\hh M$ is finite and $\Ext{}RMM$ is bounded is either perfect or quasi-isomorphic to a bounded complex of injective $R$-modules.
  \end{chunk}

In contrast with $\tor$-persistence, not every ring can be $\ext$-persistent.

\begin{chunk}
\label{ch:sd}
An $R$-complex $M$ is \emph{semidualizing} if $\hh M$ is finite and the canonical morphism $R\to \Ext{}RMM$ is bijective; equivalently, $\Ext iRMM=0$ for $i\ne 0$ and there is an isomorphism $\Ext 0RMM\cong R$.

The ring $R$ itself, viewed as an $R$-module is semidualizing, as is any dualizing complex for $R$. There exist rings that admit semidualizing complexes besides these obvious ones; see, for example, \cite[7.8]{LWC}. Such a ring is not $\ext$-persistent.
\end{chunk}

The next result contains the analogue of Proposition~\ref{prop:descent} and Theorem~\ref{thm:rigidT} for Ext-persistence; unlike in the latter, there are no additional hypothesis on $Q$.

    \begin{theorem}
  \label{thm:mainExt}
Assume that there exist a local homomorphism $R\to R'$ of finite flat dimension and a deformation
$R'\twoheadleftarrow Q$.

If $Q$ is $\ext$-persistent, then so is $R$.
    \end{theorem}

  \begin{corollary}
    \label{cor:mainExt}
If $R$ satisfies the hypotheses of Theorem \emph{\ref{thm:mainTor}}, then it is $\ext$-persistent. In particular, such an $R$ has no nontrivial semidualizing complexes.
  \end{corollary}

The proofs are given in \ref{pf:mainExt} and \ref{pf:mainExtc}, respectively, following some preparation.

The next result complements Proposition~\ref{prop:friendly}. It follows from the implication (i)$\implies$(ii) that $\tor$-friendly rings are $\ext$-persistent.

 \begin{proposition}
  \label{prop:friendly-ext}
The following conditions on a local ring $R$ are equivalent.
  \begin{enumerate}[\quad\rm(i)]
 \item
The ring $R$ is $\tor$-friendly.
 \item
If $U$ and $V$ are $R$-complexes, such that $\hh U$ and $\hh V$ are finite and $\Ext{}RUV$ is bounded, then
$U$ is perfect or $V$ is quasi-isomorphic to a bounded complex of injective $R$-modules.
 \item
If $U$ and $V$ are $R$-complexes, such that $\hh U$ and $\hh V$ have finite length and $\Ext{}RUV$ is bounded, then
$U$ is perfect or $V$ is quasi-isomorphic to a bounded complex of injective $R$-modules.
  \end{enumerate}
  \end{proposition}

  \begin{proof}
The implication (ii)$\implies$(iii) is a tautology, while (iii)$\implies$(ii) is verified by an argument
similar to the one for the corresponding implication in Proposition~\ref{prop:friendly}.

Let $E$ be the injective hull of the residue field of $R$. The standard isomorphisms
\[
\Hom_{R}(\Tor nRUV,E) \cong \Ext{n}RU{\Hom_{R}(V,E)}
  \quad\text{for}\quad
n\in\BZ
\]
show that $\Tor{}RUV$ is bounded if and only if $\Ext{}RU{\Hom_{R}(V,E)}$ is bounded.
When $\hh V$ has finite length, Matlis duality yields the following assertions:
\begin{enumerate}[\quad\rm(a)]
\item
The length of $\hh{\Hom_{R}(V,E)}$ is finite.
\item
The canonical map $V\to \Hom_{R}(\Hom_{R}(V,E),E)$ is an quasi-isomorphism.
\item
$V$ is perfect if, and only if, $\Hom_{R}(V,E)$ is quasi-isomorphic to a bounded complex of injective $R$-modules.
\end{enumerate}
In view of these properties, condition (iii) above is equivalent to the corresponding condition in Proposition~\ref{prop:friendly},
so that result gives (iii)$\iff$(i).
  \end{proof}

  \begin{chunk}
    \label{ch:Extss}
    Let $R=Q/(\bsf)$ where $Q$ is a local ring and $\bsf$ is a $Q$-regular set.

If $\Ext{}RMN$ is bounded, then so is $\Ext{}QMN$.

This follows from the standard change-of-rings spectral sequence
  \begin{align*}
\CE 2pq=\Ext pRM{\Ext qQRN}\implies\Ext{p+q}QMN
  \end{align*}
Indeed, resolving $R$ over $Q$ by means of the Koszul complex on $\bsf$ one gets an
isomorphisms $\Ext qQRN\cong N^{\binom cq}$ for every integer $q$, whence $\CE 2pq\cong\Ext pRMN^{\binom cq}$.
    \end{chunk}

The statement about projective dimensions in the following result is equivalent to \cite[4.2]{AB}, where
the proof relies on minimal free resolutions.  The argument given below works equally well for
injective dimension and for projective dimension.

  \begin{proposition}
    \label{prop:ext}
Let $R\twoheadleftarrow Q$ be a deformation and $U$ an $R$-complex with $\hh U$ finite,
such that $\Ext{2j}RUU=0$ for some positive integer $j$.
  \begin{enumerate}[\rm(1)]
    \item
If $U$ is perfect over $Q$, then it is perfect over $R$.
    \item
If $U$ is quasi-isomorphic to a bounded complex of injective $Q$-modules,
then it is quasi-isomorphic to a bounded complex of injective $R$-modules.
  \end{enumerate}
  \end{proposition}

  \begin{proof}
(2)  Let $\{\xi_1,\dots,\xi_s\}$ be a set of generators of $\Ext{}RkU$ as graded module over $R[\bschi]$;
see \ref{ch:AGP}.  For $m=\max\{\deg(\xi_r)\}_{1\les r\les s}$ and each $i\ge m+2j$ we then have
  \begin{align*}
\Ext{i}RkU
&=\sum_{h\ges j}(\bschi)^{h}\Ext{i-2h}RkU \\
&\subseteq\sum_{h\ges j}(\Ext{2}RUU)^j(\bschi)^{h-j}\Ext{i-2h}RkU \\
&\subseteq\sum_{h\ges j}\Ext{2j}RUU\circ\Ext{i-2j}RkU
  \end{align*}
Thus, $\Ext{2j}RUU=0$ implies $\Ext {\gg0}RkU=0$, so by \cite[5.5(I)]{AF} it is quasi-isomorphic to a
bounded complex of injective $R$-modules.

(1) One uses the graded bimodule $\Ext{}RUk$ and argues as above.
  \end{proof}

\begin{proposition}
\label{pr:basechange}
Let $R\to S$ be a local homomorphism of finite flat dimension and $F$ a semifree $R$-complex with $\hh F$ finite.

If the $S$-complex $S\otimes_{R}F$ is perfect, then so is $F$.

If the $S$-complex $S\otimes_{R}F$ has finite injective dimension, then so does $F$.
\end{proposition}

\begin{proof}
We may assume that the $R$-complex $F$ is minimal, that is to say, $\partial(F)\subseteq \fm F$, where $\fm$ is the maximal ideal of $R$. Then $F$ is perfect if and only if $F_{i}=0$ for $|i|=0$ for $i\gg 0$. It is evident that the $S$-complex $S\otimes_{R}F$ is minimal, so it follows that when it is perfect so is $F$.

In the remainder of the proof, let $k$ denote the the residue field of $R$ and $E$ its resolution by finite free $R$-modules. Let $I$ be a semiinjective resolution of $F$ and $G$ a finite resolution of $S$ by flat $R$-modules. One then has the following quasi-isomorphisms of $R$-complexes:
\[
G\otimes_{R}I \xla{\ \simeq\ } G\otimes_{R} F \xra{\ \simeq\ } S\otimes_{R} F\,.
\]
These induce the quasi-isomorphisms below
\begin{align*}
\Hom_{R}(E,I)\otimes_{R}G
	& \cong \Hom_{R}(E,G\otimes_{R}I) \\
	& \simeq \Hom_{R}(E,G\otimes_{R} F) \\
        & \simeq \Hom_{R}(E,S\otimes_{R} F) \\
        & \cong \Hom_{S}(S\otimes_{R}E,{S\otimes_{R}F})
\end{align*}
The isomorphisms are standard. Note that $\hh{S\otimes_{R}E}$ is bounded, for it is isomorphic to $\Tor{}RSk$ and the flat dimension of $S$ over $R$ is finite. The finiteness of the injective dimension of $S\otimes_{R}F$ thus implies that the homology of the complex $\Hom_{S}(S\otimes_{R}E,{S\otimes_{R}F})$ is bounded. Therefore the quasi-isomorphisms above yields the same conclusion for the complex $\Hom_{R}(E,I)\otimes_{R}G$. The homology of the $R$-complex $\Hom_{R}(E,I)$ is degreewise finite, as it is isomorphic to $\Ext{}RkF$, so the version of the Amplitude Inequality from \cite[3.1]{FI} (see \cite[5.12]{DGI} for a different proof) implies then that $\Ext{}RkF$ is bounded. This implies that the injective dimension of $F$ is finite, by see \cite[5.5(I)]{AF}.
\end{proof}

  \begin{chunk}\emph{Proof of Theorem} \ref{thm:mainExt}.
   \label{pf:mainExt}
Assume that $\Ext{}RMM$ is bounded.

Let $F\xra{\simeq}M$ be a free resolution by finite free $R$-modules.  The graded module $\hh{\Hom_{R}(F,M)}$ is isomorphic to $\Ext{}RMM$, and thus bounded. Let $G\xra{\simeq}R'$ be a finite resolution by flat $R$-modules.  In view of the quasi-isomorphisms
  \begin{align*}
G\otimes_R\Hom_{R}(F,M)
&\cong\Hom_{R}(F,G\otimes_RM)\\
&\xla{\simeq} \Hom_{R}(F,{G\otimes_RF})\\
&\xra{\simeq} \Hom_{R}(F,R'\otimes_R F)\\
&\cong\Hom_{R'}(R'\otimes_RF, R'\otimes_R F)
 \end{align*}
all the complexes involved have bounded homology.  Since $R'\otimes_RF$ is semifree over $R'$,
the homology of the last complex is $\Ext{}{R'}{R'\otimes_RF}{R'\otimes_RF}$.

The hypothesis is inherited by the $R'$-module $M'=R'\otimes_{R}F$ and the finiteness of $\pdim_{R'}(M')$ (respectively, $\idim_{R}(M')$) implies that of $\pdim_{R}M$ (respectively, $\idim_{R}(M')$); see Proposition~\ref{pr:basechange}. Thus, we replace $R$ and $M$ by $R'$ and $M'$, and assume that $R$ itself has a deformation $Q\to R$, where $Q$ is $\ext$-persistent.

Now from \ref{ch:Extss} one gets that $\Ext {}QMM$ is bounded, and hence that $\pdim_{Q}M$ or $\idim_{Q}M$ is finite. Since $\Ext{}RMM$ is bounded, it then follows from Proposition~\ref{prop:ext} that $\pdim_{R}M$ or $\idim_{R}M$ is finite, as desired.   \qed
  \end{chunk}

\begin{chunk}\emph{Proof of Corollary} \ref{cor:mainExt}
 \label{pf:mainExtc}
When $R$ satisfies the hypotheses of Theorem \ref{thm:mainTor}, Lemmas \ref{lem:nodef} and \ref{lem:friendly} yield
local ring homomorphisms $R\to R'\twoheadleftarrow Q$, such that the first one is of finite flat dimension, the second
one is a deformation, and the ring $Q$ is $\tor$-friendly.  Therefore $Q$ is $\ext$-persistent, by Proposition~\ref{prop:friendly-ext}, and hence so is $R$, by Theorem~\ref{thm:mainExt}.
  \qed
\end{chunk}

\section{Rings with the Auslander-Reiten property}
 \label{sec:ARP}
As before, $R$ denotes a commutative noetherian ring.

\begin{chunk}
  \label{ch:AR}
We say that $R$ has the \emph{Auslander-Reiten property} if every finite $R$-module $M$ satisfies the equality
  \begin{equation}
    \label{eq:mainExt}
\pdim_RM=\sup\{i\in\BZ\mid\Ext iRM{R\oplus M}\ne0\}
  \end{equation}

This is equivalent to the condition that the graded module $\Ext{}RM{R\oplus M}$  is bounded only if $\pdim_RM<\infty$. For when $p=\pdim_{R}M$ is finite, using a minimal free resolution of $M$ we get $\Ext pRMR\ne0$ and $\Ext nRMN=0$ for all $n>p$ and any $R$-module $M$, so that \eqref{eq:mainExt} holds.
  \end{chunk}

\begin{theorem}
\label{thm:friendly-AR}
Each $\ext$-persistent local ring has the Auslander-Reiten property.
\end{theorem}

\begin{proof}
Let $R$ be an $\ext$-persistent local ring and $M$ a finite $R$-module such that $\Ext{}RM{R\oplus M}$ is bounded. Then $\Ext{}R{R\oplus M}{R\oplus M}$ is also bounded, and hence the $\ext$-persistence of $R$ implies that $R\oplus M$, equivalently, $M$, has finite projective dimension or $R\oplus M$ has finite injective dimension. It remains to note that, in the latter case, $\idim_RR$ is finite, so $R$ is Gorenstein, and so the finiteness of $\idim_{R}M$ implies that of $\pdim_{R}M$.
\end{proof}

The next result is a direct consequence of the preceding one and Corollary~\ref{cor:mainExt}.

  \begin{corollary}
    \label{cor:AR}
If $R$ satisfies the hypotheses of Theorem \emph{\ref{thm:mainTor}}, then it has the Auslander-Reiten property. \qed
  \end{corollary}

To wrap up this discussion we record an analogue of  Theorem~\ref{thm:mainExt};  the proof is also similar, even a little easier for one does not have to contend with finiteness of injective dimension, and so is omitted.

    \begin{theorem}
Assume that there exist a local homomorphism $R\to R'$ of finite flat dimension and a deformation
$R'\twoheadleftarrow Q$.

If $Q$ has the Auslander-Reiten property, then so does $R$. \qed
    \end{theorem}

In contrast with \eqref{eq:mainExt} the projective dimension of $M$ cannot be determined from where $\Tor{}RMM$ vanishes, though the latter does provide bounds.

\begin{example}
\label{ex:Auslander}
If $M$ is a finite $R$-module of finite projective dimension, then with $s=\sup\{i\mid \Tor iRMM\ne 0\}$ there are bounds
\[
s \leq \pdim_{R} M \leq (\depth R+s)/2\,.
\]
The inequality on the left is clear, whereas one on the right follows  from \cite[2.4,2.7]{FI}.

These inequalities can be strict: Let $R$ be a regular local ring of Krull dimension $2d$. For each integer $1 \le j\le d$, the $R$-module $M_{j}=\syz[R]{d+j}k$ satisfies
\begin{align*}
\pdim_{R} M_{j} & = d - j \\
\Tor iR{M_{j}}{M_{j}} &=0 \quad\text{for $i\ge 1$}
\end{align*}
For a proof of these assertions see, for example,  \cite[2.2]{CIPW}.
\end{example}

\end{document}